\documentclass[pdftex,10pt,a4paper,twoside,twocolumn,fleqn]{article}
\pdfoutput=1
\usepackage{graphicx}
\usepackage{mathpazo}   
\usepackage[T1]{fontenc}
\usepackage{textcomp}
\usepackage{amssymb}
\usepackage{amsmath}
\usepackage{natbib}
\usepackage{subfigure}
\newcommand{\R}{\ensuremath{\mathbb{R}}}
\newcommand{\norm}[1]{\left|\left|#1\right|\right|}

\newcommand{\conv}[1]{\operatorname{conv}{#1}\,}
\usepackage[pdftex,
    bookmarks,
    bookmarksopen=true,
    bookmarksnumbered=true,
    pdfauthor={Daniel P. Mohr, Ina Stein,  Thomas Matzies, Christina A. Knapek},
    pdftitle={Robust Topology Optimization of Truss with regard to Volume},
    pdfkeywords={robust optimization, uncertainty, truss topology optimization, sizing optimization, ground structure, plastic design},
    colorlinks, linkcolor=blue, urlcolor=blue
]{hyperref}
\usepackage{anysize}
\marginsize{2cm}{2cm}{2cm}{2cm}
\usepackage{fancyhdr}
\title{Robust Topology Optimization of Truss with regard to Volume}
\author{Daniel P. Mohr\thanks{E-mail: \href{email:daniel.mohr@ttrr.cadae.de}{daniel.mohr@ttrr.cadae.de}} \and
Ina Stein\thanks{Institute of Mathematics and Computer Applications, Department of Aerospace Engineering, Universit\"at der Bundeswehr M\"unchen, Werner-Heisenberg-Weg 39, D-85577 Neubiberg, Germany.} \and
Thomas Matzies\thanks{Institute of Lightweight Structures, Department of Aerospace Engineering, Universit\"at der Bundeswehr M\"unchen, Werner-Heisenberg-Weg 39, D-85577 Neubiberg, Germany} \and
Christina A. Knapek\thanks{Max-Planck-Institut f\"ur extraterrestrische Physik, Postfach 1312, Giessenbachstr., 85741 Garching, Germany}}
\date{\today}
\usepackage{footmisc}
\DefineFNsymbols*{lamportstar}[math]{{\star}{\star\star}{\star\star\star}{\star\star\star\star}}
\setfnsymbol{lamportstar}

\usepackage[standard,thmmarks]{ntheorem}
\begin{document}
\maketitle
\section*{Abstract}
A common problem in the optimization of structures is the handling of uncertainties in the parameters. If the parameters appear in the constraints, the uncertainties can lead to an infinite number of constraints. Usually the constraints have to be approximated by finite expressions to generate a computable problem. Here, using the example of the topology optimization of a truss, a method is proposed to deal with such uncertainties by using robust optimization techniques, leading to an approach without the necessity of any approximation. With adequately chosen load cases, the final expression is equivalent to the multiple load case. Simple numerical examples of typical problems illustrate the application of the method.
\section{Introduction}
Since the first work by \cite{michell}, a great number of methods for the optimization of truss geometries were developed. A good overview of methods is given in \cite{bendsoebentalzowe}. A possible way to deal with uncertainties in parameters during the optimization process is robust optimization, a detailed survey of which can be found in \cite{Beyer_Sendhoff_2007}. The idea behind the approach chosen in this paper for the robust topology optimization of a truss --- following the ground structure method \citep{dorngomorygreenberg} with the use of a matrix force method \citep{przemieniecki} --- is outlined in the following.

A general optimization problem $\min_x{z(x)}$ subject to $x \in S$ is often given for various different parameters, and principally represents the structure of the problem. For example, the sizing optimization of a truss determines the overall structure. The meaning of the structure arises from the specification of actual data, for instance the topology of the bars, the definition of the supports as well as the loads. In general, one has the optimization problems $\min_x{z(x,y)}$ subject to $x \in S(y)$ for different fixed parameters $y$ defined by the actual data.

If the parameters are not exactly known, one has to consider $\min_x{z(x,y)}$ subject to $x \in S(y)$ for $y \in Y$. Here, $Y$ represents the set of the uncertain parameters. Even without stating precisely the origin of this uncertainties (or scattering) in the parameters, the question arises on how to obtain a computable optimization problem.

The reasons for uncertainties in the parameters can be e.g. measurement errors if the parameters are actually obtained as or derived from measured variables. A parameter could also scatter across intervals which are restricted by limits provided by the user. Later it will be shown that it can make sense to choose an artificial limit for the parameters. 

The importance of uncertainties arises simply due to their occurrence as e.g. measurement errors in real live applications.
Measurement errors can basically be divided into two groups: systematic and random errors \citep{taylor}. Systematic errors are caused for instance by bad calibration of instruments or defective design. The distinctive feature of this kind of error is that it is always directed, and its impact is not necessarily known, which makes it difficult to treat in the error analysis (a miscalibration can cause an offset of a measured parameter, and this is not a scattering), and the best strategy is to identify and remove all sources for systematic errors.

Random errors originate from a multitude of sources. Common sources are the thermal noise superimposed on a measurement due to a finite measurement device temperature, statistical fluctuations in counting experiments, or limited resolution of instruments yielding upper and lower limits for a parameter, to name a few examples. Random errors can usually not be avoided, but if their nature is known, they can be treated in a defined manner.

If uncertainties follow a known statistical distribution, this knowledge can be integrated in procedures handling the parameters.
For example in the case of the calculation of a mean value from a large number of independent measurements, a Gaussian distribution is assumed to justify the average as the most probable value. Other well-known and common distributions are the Poisson distribution in counting experiments, or the equipartition of the dice roll.

If the distribution is completely unknown due to the lack of a justifiable model, boundaries of the interval in which the measured variable lies have to be established to generate a treatable problem. This is the equivalent to a user supplied limit for a parameter.

In the case of an either unknown distribution or if only the limits are known, one could choose an arbitrary nominal value $y_0 \in Y$ and solve the problem for it. This procedure can lead to results which are invalid for any other value drawn from $Y$. More reasonable is the consideration of all parameter values $y \in Y$ - matching the worst-case-approach. This immediately leads to the robust optimization \citep{ben_tal_nemirovski_robust_convex_optimization} constituting the main body of the work presented here.

Even in the case of a known stochastic distribution, a robust optimization can be carried out by the interval boundaries if required, although then the information of the distribution is neglected, which can be considered by more reasonable procedures. The expectation value can be used to specifically choose a nominal value and solve the problem with respect to that value, but again accompanied by the loss of a large part of the available information. Due to the characteristics of the expectation value the approach considered here constitutes an average-case; however, the expectation value is just one special value, and therefore possibly leads to a solution valid for one point only. An example for the difficulty of this single-value approach is given by the problem of constructing a building at the coast under the assumption of a typical (ever-present) wind load (e.g. the expectation value), but neglecting the possibility of instantaneous absence of wind, causing the building to collapse if the latter case occurs. Another possibility is to consider only some very probable cases (larger than same given probability $p_0$) and thus to disregard the improbable cases. For very small $p_0$ this is called reliability based design optimization (RBDO), while for larger $p_0$ the denomination is robust design \citep{doltsinis}. Formally, this leads to $\min_x{z(x,y)} \mbox{ s.t. } x \in \cap_{P(y) \geq p_0}{S(y)}$. The robust optimization is again a direct implication. Nonetheless one should not neglect the sometimes serious consequences of unlikely cases, even if they seem to be statistically improbable in theory. Therefore, this approach does not in general lead to safe-life, fail-safe or damage tolerance.

Another possibility for the source of uncertainties are application errors by the user. With the assumption of a range or an order of magnitude of these uncertainties this case can be treated as above and leads again to robust optimization.

First, we will look at the expected numerical behavior (condition) of topology optimization of truss \citep{przemieniecki} in view of the topology to obtain artificial limits for the robust optimization. In \eqref{lineare Topologieoptimierung Volumen: Zielfunktion}--\eqref{lineare Topologieoptimierung Volumen: Boxbedingung der Stabdicke} we will present the typical topology optimization, which will be extended to a multiple load case in \eqref{lineare Topologieoptimierung Volumen: robuste Zielfunktion}--\eqref{lineare Topologieoptimierung Volumen: robust zulaessiger Zulaessigkeitsbereich diskret approximiert} by a --- for the robust optimization --- unnecessary approximation. We then start over with \eqref{lineare Topologieoptimierung Volumen: multiple load case}--\eqref{lineare Topologieoptimierung Volumen: zulaessigkeitsbereich multiple load case} as the new initial problem to reach its robust counterpart \eqref{lineare Topologieoptimierung Volumen: robuste Zielfunktion mehrfacher Lastfall}--\eqref{lineare Topologieoptimierung Volumen: robust zulaessiger Zulaessigkeitsbereich mehrfacher Lastfall}, which will be solved without approximation. The robust counterpart \eqref{lineare Topologieoptimierung Volumen: robuste Zielfunktion mehrfacher Lastfall}--\eqref{lineare Topologieoptimierung Volumen: robust zulaessiger Zulaessigkeitsbereich mehrfacher Lastfall} is a semi-infinite optimization problem. By looking at the finite numbers of design variables, namely the cross sections of the bars, we will be able to bypass the infinite numbers of constraints and obtain a linear program \eqref{lineare Topologieoptimierung Volumen: robuste Zielfunktion mehrfacher Lastfall finite number of constraints}--\eqref{lineare Topologieoptimierung Volumen: robust zulaessiger Zulaessigkeitsbereich mehrfacher Lastfall finite number of constraints}, which provides a worst-case or robust solution.

\section{Condition of Topology Optimization of Truss}\label{section Condition of Topology Optimization of Truss}
\begin{figure}
\hspace*{\fill}
\subfigure{\label{fig:1}\thesubfigure\includegraphics[width=0.21\textwidth]{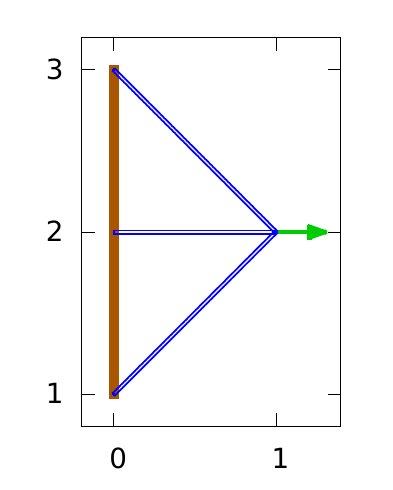}}
\hspace*{\fill}
\subfigure{\label{fig:2}\thesubfigure\includegraphics[width=0.21\textwidth]{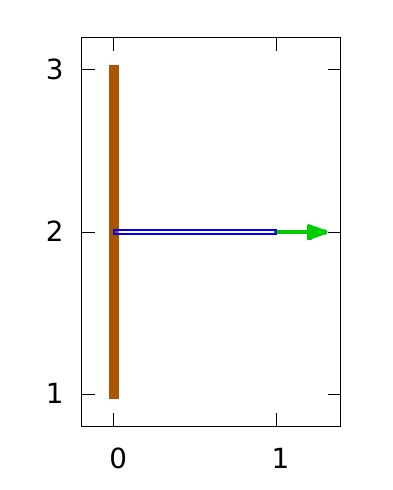}}
\hspace*{\fill}
\caption{\subref{fig:1} Ground structure with fixed nodes on the left side and arrow representing the force (example $1$). \subref{fig:2} Optimal solution.}
\label{fig:1,2}
\end{figure}

\begin{figure}
\hspace*{\fill}
\subfigure{\label{fig:3}\thesubfigure\includegraphics[width=0.21\textwidth]{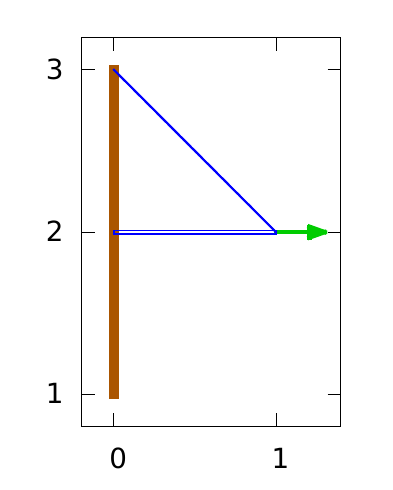}}
\hspace*{\fill}
\subfigure{\label{fig:4}\thesubfigure\includegraphics[width=0.21\textwidth]{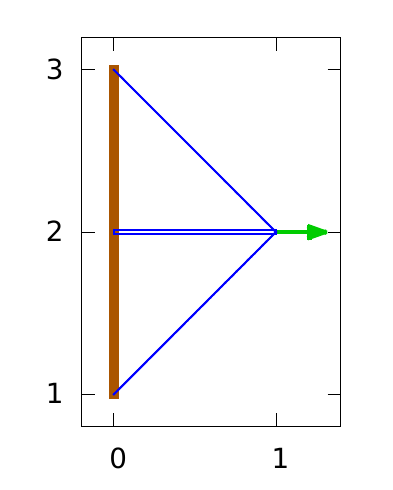}}
\hspace*{\fill}
\caption{\subref{fig:3} Optimal solution of example $1$ (c.f. Fig. \ref{fig:1}) for the case with a perturbed force. \subref{fig:4} Robust optimal solution of example $1$.}
\label{fig:3,4}
\end{figure}

The simple example of a $2$-dimensional topology optimization of a plain frame (cf. \eqref{lineare Topologieoptimierung Volumen: Zielfunktion}--\eqref{lineare Topologieoptimierung Volumen: Boxbedingung der Stabdicke}) on the ground structure of the $3$ bars $v_i = \left(e_i^{(1)},e_i^{(2)}\right)$ with the nodes $e_i^{(1)} = (0,i)^T \in \R^2$ and $e_i^{(2)} = (1,2)^T \in \R^2$ for $i=1,2,3$, which are fixed in $e_i^{(1)}$, $i=1,2,3$, demonstrates the bad condition concerning a perturbation of the force direction in the case of a horizontal force acting on the node $(1,2)^T$ with regard to the topology (Fig. \ref{fig:1}).

For the case of an unperturbed horizontal force of quantity $f\in\R$ acting in positive $x$-direction, the optimal solution, illustrated in Fig. \ref{fig:2}, yields the thickness of the bars as
\begin{align}
\tilde{s}^{(1)} = \left(\begin{array}{ccc}0&\frac{f}{\sigma_+}&0\\\end{array}\right)^T,
\end{align}
where $\sigma_-<0<\sigma_+$ represents the stress limits (e.g. lower and upper yield point), which is determined equally for all bars.
For a perturbation in positive $y$-direction, the force $\tilde{f}$ is now constructed with the components $\tilde{f}_x > \tilde{f}_y$; the magnitude of the force $\norm{\left(\begin{array}{cc}\tilde{f}_x&\tilde{f}_y\\\end{array}\right)} = f$ remains unchanged. As one optimal solution for $-\sigma_- \leq \sigma_+$ the thickness of the bars results in
\begin{align}
\tilde{s}^{(2)} = \left(\begin{array}{ccc}\frac{\sqrt{2} \tilde{f}_y}{\sigma_+}&\frac{\tilde{f}_x - \tilde{f}_y}{\sigma_+}&0\\\end{array}\right)^T,
\end{align}
which represents another topology than the unperturbed case (Fig. \ref{fig:3}). Since the optimization problem is a linear program, there are typically many solutions with the same objective function value. The stated solution $\tilde{s}^{(2)}$ is an edge of the feasible region. Obviously, another edge is $\tilde{s}^{(3)} = \left(\begin{array}{ccc}0&\frac{\tilde{f}_x - \tilde{f}_y}{\sigma_+}&\frac{\sqrt{2} \tilde{f}_y}{\sigma_+}\\\end{array}\right)^T$ and every point between these solutions is a solution, too.

Apparently, here the solution of the unperturbed case is not feasible for the perturbed case. The naming convention of the perturbed and unperturbed case is of philosophical nature. Therefore, a numerically calculated solution can in general not be accepted as feasible for the analytically posed problem --- this means bad condition. As robust optimization yields a solution which is feasible for a box including the original parameter and also the backward error caused by the backward error analysis performed for the stability of a specific algorithm, it provides a possible resource with artificially chosen limits.

\section{Robust Optimization}
The basic idea of robust optimization, as it will be used later, is outlined in the following.
For the general optimization problem $\min_x{z(x,y)}$ subject to $x \in S(y)$ restricted to the fixed parameter $y = y_0$ with the objective function $z$, which is  without loss of generality real-valued, and the parameter-dependent search space $S(y)$, a robust solution in the case of perturbations in the parameter $y$ is searched.
The perturbations constitute the set $Y$. Obviously, one has to distinguish between objective function and constraints. In the sense of the worst-case-consideration only {\em robust feasible} solutions $x \in S(Y) := \cap_{y \in Y}{S(y)}$ are considered. Clearly in general $S(Y)$ could be empty and the problem is robust unfeasible. A {\em robust optimal} solution is a robust feasible $x$ which minimizes the worst-case objective function. With that follows the robust optimal solution subject to robust feasibility as solution of $\min_{x,t}{t}$ subject to $\forall y \in Y: z(x,y) \leq t$ and $x \in S(Y)$, cf. \cite{ben_tal_nemirovski_robust_convex_optimization}. In the linear case this robust optimization agrees also with \cite{soyster_convex_programming_with_set_inclusive_constraints_and_appolications_to_inexcat_linear_programming}. If $y$ is not present in the constraints, then the robust optimization problem is equivalent to the well-known min-max problem originating from the game theory \cite[chapter III]{neumann_morgenstern}.

\section{Topology Optimization of Truss with regard to Volume (single load case)}
This general concept of robust optimization can now be applied to the topology optimization.

Here, we do not treat the solvability.
Instead we request a reasonable problem and leave the question of the solvability to the user who provided it.

The analytically stated and known (cf. \cite[matrix force method]{przemieniecki}; \cite{marti_stochastic_optimization_methods_in_optimal_engineering_design_under_stochastic_uncertainty}) linear problem of the topology optimization of a truss with respect to its volume (plastic design) is given by the objective function to be minimized
\begin{align}
\min_{s,w}{l^T s}\label{lineare Topologieoptimierung Volumen: Zielfunktion}
\end{align}
with the bar lengths $l \in \R^n$ subject to the equilibrium condition
\begin{align}
C w = f \label{lineare Topologieoptimierung Volumen: Gleichgewichtsbedingung}
\end{align}
with the reduced geometry matrix $C \in \R^{m \times n}$, the inner bar forces $w \in \R^n$  and the reduced applied external loads $f \in \R^m$. Further given are the constraints
\begin{align}
\sigma_- s \leq w \leq \sigma_+ s \label{lineare Topologieoptimierung Volumen: linear elasticity by the stress limits}
\end{align}
for the linear elasticity by --- for all bars equal --- stress limits $0 > \sigma_- \in \R$ and $0 < \sigma_+ \in \R$ (e.g. yield points for pressure and tension of the material all bars are made of), and the box constraint
\begin{align}
0 \leq s \leq s_{\max} \label{lineare Topologieoptimierung Volumen: Boxbedingung der Stabdicke}
\end{align}
for the cross sections of the bars $s \in \R^n$ with the given maximal bar cross section $s_{\max} \in \R^n$.
Here, a maximum bar cross section adjusted to the bar length allows with \eqref{lineare Topologieoptimierung Volumen: linear elasticity by the stress limits} a prevention of buckling \citep{przemieniecki}. The practical problem also includes the supports able to absorb an arbitrary load. These support positions are removed from the equilibrium condition by canceling the respective lines, and it remains the reduced form \eqref{lineare Topologieoptimierung Volumen: Gleichgewichtsbedingung} with the reduced loads and the reduced geometry matrix. Here, we restrict ourselves to this reduced forms, and consequently abandon the explicit representation of the supports.

In calculations relevant to the praxis, the stress limits $\sigma_-$ and $\sigma_+$ are often reduced by a safety factor (i.e. $0.7$). The geometry matrix $C$ describes the ground structure chosen in the modeling procedure. Hence, the only real parameter are the external loads $f$. As already shown in chapter \ref{section Condition of Topology Optimization of Truss}, the topology calculated from the optimization problem \eqref{lineare Topologieoptimierung Volumen: Zielfunktion}--\eqref{lineare Topologieoptimierung Volumen: Boxbedingung der Stabdicke} is bad conditioned with regard to the force direction. Therefore, now a substitute problem for \eqref{lineare Topologieoptimierung Volumen: Zielfunktion}--\eqref{lineare Topologieoptimierung Volumen: Boxbedingung der Stabdicke} is sought, which is robust against perturbations $f \in F \subset \R^m$.

For a robust feasible solution the constraints have to be fulfilled for all $\tilde{f} \in F$:
\begin{align}
\forall \tilde{f} \in F : \; & C w(\tilde{f}) = \tilde{f} \label{lineare Topologieoptimierung Volumen: robust zulaessiger Zulaessigkeitsbereich}\\
&\sigma_- s \leq w(\tilde{f}) \leq \sigma_+ s \\
&0 \leq s \leq s_{\max} \label{lineare Topologieoptimierung Volumen: robust zulaessiger Zulaessigkeitsbereich ende}
\end{align}

Typically, the number of elements of $F$ is infinite and therefore the number of constraints as well. Since the objective function \eqref{lineare Topologieoptimierung Volumen: Zielfunktion} is independent of $\tilde{f}$ and $w(\tilde{f}) \in \R^n$, the robust optimization problem \eqref{lineare Topologieoptimierung Volumen: Zielfunktion} subject to \eqref{lineare Topologieoptimierung Volumen: robust zulaessiger Zulaessigkeitsbereich}--\eqref{lineare Topologieoptimierung Volumen: robust zulaessiger Zulaessigkeitsbereich ende} arises.

For a discrete approximation $\tilde{F} = \left\{f_1,f_2,\ldots,f_{n_{\tilde{F}}}\right\} \subset \R^m$ of $F$ with an at this point unspecified choice of $f_{i}$, follows the structure
\begin{align}
\min_{s,w}{l^T s}\label{lineare Topologieoptimierung Volumen: robuste Zielfunktion} &&\\
&\mbox{ s.t. } && \forall f_i \in \tilde{F}: & \nonumber\\
&&& \qquad C w_i = f_i \in \R^m \\
&&& \qquad \sigma_- s \leq w_i \leq \sigma_+ s\\
&&& 0 \leq s \leq s_{\max} \label{lineare Topologieoptimierung Volumen: robust zulaessiger Zulaessigkeitsbereich diskret approximiert}
\end{align}
of a linear program with the unknowns $s \in \R^n$ and $w_i \in \R^n$, also known as multiple load case. 

In the step to \eqref{lineare Topologieoptimierung Volumen: robuste Zielfunktion}--\eqref{lineare Topologieoptimierung Volumen: robust zulaessiger Zulaessigkeitsbereich diskret approximiert} we approximated the set $F$ by $\tilde{F}$ without stating how this step is performed technically. This is in no case trivial! But manually choosing adequate discrete data from $F$ to end up at $\tilde{F}$ might be legitimate in a preliminary design, which is what the topology optimization often only provides. This manual choice is likewise not trivial. The perturbations yielding $F$ have to be considered as a small neighborhood of $f$. A reasonable approximation $\tilde{F}$ lies also in a small neighborhood of the applied $f$. However, by this naive approximation we were able to arrive at a computationally tractable problem \eqref{lineare Topologieoptimierung Volumen: robuste Zielfunktion}--\eqref{lineare Topologieoptimierung Volumen: robust zulaessiger Zulaessigkeitsbereich diskret approximiert} with the structure of the multiple load case.

Nevertheless, with the technique presented in this paper one does not have to solve the problem of the approximation of the set $F$. In the following we are even able to treat a generalization of the single load case \eqref{lineare Topologieoptimierung Volumen: Zielfunktion}--\eqref{lineare Topologieoptimierung Volumen: Boxbedingung der Stabdicke}.

\section{Topology Optimization of Truss with regard to Volume (multiple load case)}
We now want to consider the multiple load case, which is apparently a natural extension of \eqref{lineare Topologieoptimierung Volumen: Zielfunktion}--\eqref{lineare Topologieoptimierung Volumen: Boxbedingung der Stabdicke}.
Instead of just one single load $f$, we look at a given but arbitrary set of loads $\tilde{F}_{mlc} = \left\{f_1,f_2,\ldots,f_{n_f}\right\} \subset \R^m$. This leads to the optimization problem:
\begin{align}
\min_{s,w}{l^T s}\label{lineare Topologieoptimierung Volumen: multiple load case} &&\\
&\mbox{ s.t. } && \forall f_i \in \tilde{F}_{mlc}: \nonumber\\
&&& \qquad C w_i = f_i \\
&&& \qquad \sigma_- s \leq w_i \leq \sigma_+ s\\
&&& 0 \leq s \leq s_{\max} \label{lineare Topologieoptimierung Volumen: zulaessigkeitsbereich multiple load case}
\end{align}
In contrast to $\tilde{F}$, the distance between two elements of $\tilde{F}_{mlc}$ cannot be assumed as small. We can again generate a robust optimization against given perturbations $\hat{F}$. With $\tilde{F}_{mlc} \subset \hat{F} \subset \R^m$ follows as a substitution problem with a typically infinite number of elements:
\begin{align}
\min_{s}{l^T s}\label{lineare Topologieoptimierung Volumen: robuste Zielfunktion mehrfacher Lastfall} &&\\
&\mbox{ s.t. } && \forall \hat{f} \in \hat{F}: \nonumber\\
&&& \qquad C w(\hat{f}) = \hat{f} \\
&&& \qquad \sigma_- s \leq w(\hat{f}) \leq \sigma_+ s\\
&&& 0 \leq s \leq s_{\max} \label{lineare Topologieoptimierung Volumen: robust zulaessiger Zulaessigkeitsbereich mehrfacher Lastfall}
\end{align}

Assuming that $\hat{F}$ is a union of $n_f$ sets, i.e. $\hat{F} := \bigcup_{i = 1}^{n_f}{\hat{F}_i}$, we will not need to approximate $\hat{F}$, but rather we are able to specify an optimization problem which is equivalent to \eqref{lineare Topologieoptimierung Volumen: robuste Zielfunktion mehrfacher Lastfall}--\eqref{lineare Topologieoptimierung Volumen: robust zulaessiger Zulaessigkeitsbereich mehrfacher Lastfall}. Of course this procedure covers the special case of the single load case. The usual design of the limits as they are provided by the engineer leads in the following to the assumption of $\hat{F}$ as the union of parallelotopes. The following is easily reproducible also for $\hat{F}$ as the union of general sets --- but both the treatment and the solvability in a following optimization algorithm depend on the structure of these sets.

Now, we examine the multiple load case which is given by the set $\tilde{F}_{mlc} = \left\{f_1, f_2, \ldots, f_{n_f}\right\}$ of $n_f$ elements. Here, the force $f_i$ acts in the nodal points given by the set $J_i$. We will describe the corresponding perturbed forces $\hat{f}_i \in \hat{F}_i \subset \R^m$ by box constraints concerning the nodal points at which $f_i$ is acting:
\begin{align}
\hat{F}_i := & \bigg\{\hat{f}_i : \hat{f}_i = f_i + \sum_{j \in J_i}^{}{\sum_{k=1}^{d}{r(j,k) \delta_{i,j,k}}}; \bigg.\\
&\bigg. \Delta_-(i,j,k) \leq \delta_{i,j,k} \leq \Delta_+(i,j,k)\bigg\}
\end{align}

Here, $d \in \{2,3\}$ specifies the dimension of the space in which the structure lies. $r(j,k) \in \R^m$ represents a vector with the $j$th component equal $1$ for the $j$th node in the spatial direction $k$, and $0$ elsewhere. $\Delta_-(i,j,k) \in \R$ and $\Delta_+(i,j,k) \in \R$, respectively, represent the perturbation in the $j$th node for the $i$th force in the spatial direction $k$. In contrast to \cite{bentalnemirovskirobusttrusstopologydesignviasemidefiniteprogramming}, we consider the perturbations only in the nodes where the forces are acting. By choosing appropriate sets $J_i$, perturbations could also be considered in any node, but here we want to obtain a robustness with regard to the parameter force.

To simplify the description, $f_i^{(j)}$, $1 \leq j \leq n_i$, $1 \leq i \leq n_f$ are introduced as the $n_i$ edges of the parallelotope $\hat{F}_i$ for the $i$-th load case, so that the convex polytopes $\hat{F}_i$ can be written as the convex hull
\begin{align*}
\hat{F}_i = \conv{\left\{ f_i^{(j)} : 1 \leq j \leq n_i \right\}}
\end{align*}
of their edges. This allows to write every $f \in \hat{F}_i$ as the convex combination
\begin{align*}
f = \sum_{j=1}^{n_i}{\alpha_j f_i^{(j)}}, \sum_{j=1}^{n_i}{\alpha_j} = 1
\end{align*}
of the edges $f_i^{(j)}$ for some appropriately chosen $\alpha_j$.

For every load case, we define $S_i := \left\{ f_i^{(j)} : 1 \leq j \leq n_i \right\}$ as the set of the edges. Furthermore, we have
\begin{align*}
\hat{F} = \bigcup_{i = 1}^{n_f}{\hat{F}_i} = \bigcup_{i = 1}^{n_f}{\conv{S_i}}
\end{align*}
the set of all loads.

To formulate the robust semi-infinite optimization problem \eqref{lineare Topologieoptimierung Volumen: robuste Zielfunktion mehrfacher Lastfall}--\eqref{lineare Topologieoptimierung Volumen: robust zulaessiger Zulaessigkeitsbereich mehrfacher Lastfall} with a finite number of constraints, we need the following theorem:

\begin{theorem}\label{Satz M1 = M2}
Let
\begin{align*}
M_1 := \Big\{\Big. s \in \R^n : \forall \hat{f} \in \hat{F} :& \exists w(\hat{f}) : \\
&C w(\hat{f}) = \hat{f} \\
& \sigma_- s \leq w(\hat{f}) \leq \sigma_+ s \Big.\Big\}
\end{align*} and
\begin{align*}
M_2 := \Big\{\Big. s \in \R^n : \forall f_j \in \bigcup_{i = 1}^{n_f}{S_i} :& \exists w_j:\\
&C w_j = f_j \\
& \sigma_- s \leq w_j \leq \sigma_+ s \Big.\Big\}.
\end{align*}
Then
\begin{align*}
M_1 = M_2.
\end{align*}
\end{theorem}

\begin{proof}
We have to show $M_1 \subseteq M_2$ and $M_2 \subseteq M_1$.

In the first case $M_1 \subseteq M_2$, let $s \in M_1$ be arbitrary. For an arbitrary $f_j \in \bigcup_{i = 1}^{n_f}{S_i}$ holds:
\begin{align*}
f_j \in \hat{F} \stackrel{s \in M_1}{\Rightarrow} : \exists w(f_j) : \sigma_- s \leq w(f_j) \leq \sigma_+ s
\end{align*}
Since $f_j$ was arbitrary, for every $f_j$ exists $w(f_j)$. This means $s \in M_2$.

To proof the second case $M_2 \subseteq M_1$, we choose an arbitrary $s \in M_2$. Let $\hat{f} \in \hat{F}$, then there exists $1 \leq i \leq n_{f}$ with $\hat{f} \in \hat{F_i} = \conv{S_i}$, and we can write $\hat{f}$ as the convex combination
\begin{align*}
\hat{f} = \sum_{j=1}^{n_i}{\alpha_j f_j^{(j)}}
\end{align*}
for appropriate $\alpha_1, \alpha_2, \ldots, \alpha_{n_i}$ with $\sum_{j=1}^{n_i}{\alpha_j} = 1$. As $s \in M_2$, there exists $w_j$:
\begin{align}
\exists w_1,\dots,w_{n_i} : & C w_j=f_i^{(j)} \label{beweis M1=M2 a} \\
&\sigma_- s \leq w_j \leq \sigma_+ s \label{beweis M1=M2 b}
\end{align}
So that for $w(\hat{f}) := \sum_{j=1}^{n_i}{\alpha_j w_j}$ holds
\begin{align}
C w(\hat{f}) & = C\sum_{j=1}^{n_i}{\alpha_j w_j} \label{beweis M1=M2 c} \\
& = \sum_{j=1}^{n_i}{\alpha_j f_i^{(j)}} = \hat{f}
\end{align}
and
\begin{align}
w & = \sum_{j=1}^{n_i}{\alpha_j w_j} \\
& \stackrel{\eqref{beweis M1=M2 b}}{\leq} \sum_{j=1}^{n_i}{\alpha_i \sigma_+ s} = \left(\sum_{j=1}^{n_i}{\alpha_i}\right) \sigma_+ s = \sigma_+ s \\
w & = \sum_{j=1}^{n_i}{\alpha_j w_j} \\
& \stackrel{\eqref{beweis M1=M2 b}}{\geq} \sum_{j=1}^{n_i}{\alpha_i \sigma_- s} = \sigma_- s. \label{beweis M1=M2 d}
\end{align}
Since $\hat{f}$ was arbitrary, \eqref{beweis M1=M2 c}--\eqref{beweis M1=M2 d}  holds for all $\hat{f} \in \hat{F}$. This means $s \in M_1$.
\end{proof}

With $S := \bigcup_{i = 1}^{n_f}{S_i}$ we can now reformulate the semi-infinite optimization problem \eqref{lineare Topologieoptimierung Volumen: robuste Zielfunktion mehrfacher Lastfall}--\eqref{lineare Topologieoptimierung Volumen: robust zulaessiger Zulaessigkeitsbereich mehrfacher Lastfall} with a finite number of constraints:
\begin{align}
\min_{s,w}{l^T s}\label{lineare Topologieoptimierung Volumen: robuste Zielfunktion mehrfacher Lastfall finite number of constraints} &&\\
&\mbox{ s.t. } && \forall f_j \in S: \nonumber\\
&&& \qquad C w_j = f_j \\
&&& \qquad \sigma_- s \leq w_j \leq \sigma_+ s\\
&&& 0 \leq s \leq s_{\max} \label{lineare Topologieoptimierung Volumen: robust zulaessiger Zulaessigkeitsbereich mehrfacher Lastfall finite number of constraints}
\end{align}
To get a typical linear program, we consider the implicit variables $w$ as design variables. Theorem \ref{Satz M1 = M2} ensures equal polyhedrons in \eqref{lineare Topologieoptimierung Volumen: robuste Zielfunktion mehrfacher Lastfall}--\eqref{lineare Topologieoptimierung Volumen: robust zulaessiger Zulaessigkeitsbereich mehrfacher Lastfall} and \eqref{lineare Topologieoptimierung Volumen: robuste Zielfunktion mehrfacher Lastfall finite number of constraints}--\eqref{lineare Topologieoptimierung Volumen: robust zulaessiger Zulaessigkeitsbereich mehrfacher Lastfall finite number of constraints} and thus theses linear programs are equivalent.

Considering a more general $\hat{F}$ as the union of convex sets $\hat{F}_i$, the convex hull of the extreme points --- in a linear case the extreme points are the edges --- describes the same sets. Therefore, for finite numbers of extreme points of the sets the above approach can be performed in the same way.

\section{Numerical Examples}
In all examples we use as material an aluminium with a yield strength of $10^8 \, \mbox{Pa}$ and a density of $2.7 \cdot 10^3 \, \frac{\mbox{\footnotesize kg}}{\mbox{\footnotesize m}^3}$. The measurement unit for the sizes of the design spaces is meter.

All these data are only scaling factors in the linear program and do not affect the resulting topology. For example let us consider $\tilde{s}$ as a solution of:
\begin{align*}
\min_{s,w}{l^T s} && \mbox{ s.t.: } & C w = f; \sigma_- s \leq w \leq \sigma_+ s; 0 \leq s
\end{align*}
Then for $0 < \alpha, \beta \in \R$ a solution of
\begin{align*}
\min_{s,w}{l^T s} && \mbox{ s.t.: } & C w = \alpha f; \beta \sigma_- s \leq w \leq \beta \sigma_+ s; 0 \leq s
\end{align*}
is $\frac{\alpha}{\beta} \tilde{s}$ with the same topology. The equation $C w = \alpha f$ is a short hand for $\tilde{C} w = \tilde{f}$ with a scaled design space represented by the scaled geometry matrix $\tilde{C} = \alpha_1 C$, a scaled force $\tilde{f} = \alpha_2 f$ and $\alpha = \frac{\alpha_2}{\alpha_1}$, $0 < \alpha, \alpha_1, \alpha_2 \in \R$. The same holds for the inequation.

The cross section size of all bars is scaled for the visualization, but in Figs. \ref{fig:1,2}, \ref{fig:3,4}, \ref{fig:5-7}, \ref{fig:8,9}, \ref{fig:10,11} and \ref{fig:12,13} the same scaling is used. The Figs. \ref{fig:12,13pov} and \ref{fig:14-16} have the same scaling as to each other, too.

The calculation is done by the library from \cite{glpk}. We used the simplex solver as well as the interior point solver.

\paragraph{Example $1$:}
In the following first example (c.f. Figs. \ref{fig:1,2}, \ref{fig:3,4}), the maximum quantity of all loads is equal $10^4 \, \mbox{N}$. For some chosen values, the initial examples result in the numerical solutions $\tilde{s}^{(1)}$ for a strictly horizontal force (Fig. \ref{fig:2}) and $\tilde{s}^{(2)}$, shown in Fig. \ref{fig:3}, for $\tilde{f}_y = 0.1 \tilde{f}_x$:
\[ \tilde{s}^{(1)} = \left(\begin{array}{c}0.00\\1.00\cdot10^{-4}\\0.00\\\end{array}\right); \quad
\tilde{s}^{(2)} = \left(\begin{array}{c}1.41\cdot10^{-5}\\8.96\cdot10^{-5}\\0.00\\\end{array}\right) \]

For a symmetric perturbation of the horizontal force of $10 \, \mbox{\%}$ in every direction the solution of \eqref{lineare Topologieoptimierung Volumen: robuste Zielfunktion mehrfacher Lastfall finite number of constraints}--\eqref{lineare Topologieoptimierung Volumen: robust zulaessiger Zulaessigkeitsbereich mehrfacher Lastfall finite number of constraints} is the robust solution $\tilde{s}^{(3)}$, illustrated in Fig. \ref{fig:4}:
\[ \tilde{s}^{(3)} = \left(\begin{array}{c}7.81\cdot10^{-06}\\1.21\cdot10^{-04}\\7.81\cdot10^{-06}\\\end{array}\right) \]

The optimal values for the objective function are the volumes in $\mbox{m}^2$:
\begin{center}\begin{tabular}{|c|c|c|}\hline
optimal & perturbed & robust \\
$l^T \tilde{s}^{(1)}$ & $l^T \tilde{s}^{(2)}$ & $l^T \tilde{s}^{(3)}$ \\\hline
$0.000100$ & $0.000109$ & $0.000144$ \\\hline
\end{tabular}\end{center}

\paragraph{Example $2$:}
The prime example of a truss from \cite{bentalnemirovskirobusttrusstopologydesignviasemidefiniteprogramming}, there calculated by semi-definite programming with regard to robustness based on a finite element approach, is illustrated in the figures \ref{fig:5}--\ref{fig:7}, but now calculated with our robust optimization method. Here, the maximal quantities of the perturbed forces are equal to the unperturbed forces.

Figure \ref{fig:5} shows the ground structure with the fixed nodes on the left side. The arrows represent the forces of equal quantities of $10^4 \, \mbox{N}$. Figure \ref{fig:6} displays the optimal solution with a volume of $0.000800 \, \mbox{m}^2$. The robust optimal solution with a volume of $0.001309 \, \mbox{m}^2$ for a symmetric perturbation of $10 \, \mbox{\%}$ in every direction of every force is presented in figure \ref{fig:7}.

The results of the topologies are congruent to the results in \cite{bentalnemirovskirobusttrusstopologydesignviasemidefiniteprogramming}. Also, in \cite{calafiorefabrizio} the same topology is obtained by a sampling-based approximate solution of the worst-case.
\begin{figure*}
\hspace*{\fill}
\subfigure{\label{fig:5}\thesubfigure\includegraphics[width=0.3\textwidth]{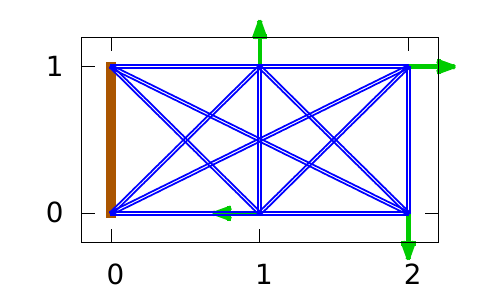}}
\hspace*{\fill}
\subfigure{\label{fig:6}\thesubfigure\includegraphics[width=0.3\textwidth]{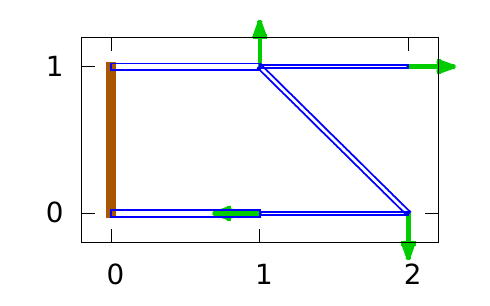}}
\hspace*{\fill}
\subfigure{\label{fig:7}\thesubfigure\includegraphics[width=0.3\textwidth]{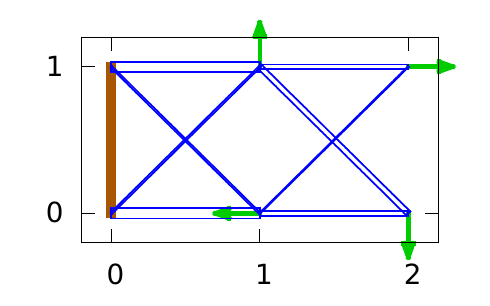}}
\caption{\subref{fig:5} Ground structure of example $2$ with fixed nodes on the left side and arrows representing the forces. \subref{fig:6} Optimal solution of the example $2$ shown in Fig. \ref{fig:5}. \subref{fig:7} Robust optimal solution of example $2$ (c.f. Fig. \ref{fig:5}).}
\label{fig:5-7}
\end{figure*}

\paragraph{Example $3$:}
We calculated a $3$-dimensional example with $27$ nodes $\left(a,b,c\right) \in \R^3$, $a,b,c \in \{1,2,3\}$ measured in meter and $274$ potential bars between every two not fixed nodes --- long bars which are located alongside several shorter bars are ignored. The supports are in the nodes $\left(1,a,b\right)$, $a,b \in \{1,2,3\}$ which are fixed in every direction. In the node $\left(3,2,1\right)$ acts a force of $4\cdot10^4 \, \mbox{N}$ in the negative $z$-direction. The optimal solution with a volume of $0.0024 \, \mbox{m}^3$ is presented in Fig. \ref{fig:8,9}. For a perturbation of $10 \, \mbox{\%}$ in every direction of the force, and scaling to obtain the same maximal quantities as in the unperturbed case, our robust optimization method produces the robust optimal solution shown in Fig. \ref{fig:10,11} with a volume of $0.0026 \, \mbox{m}^3$.

Both Fig. \ref{fig:8,9} and Fig. \ref{fig:10,11} show two solutions which are caused by linear programming. The calculation with a simplex algorithm results in edges of the polyhedron shown in Fig. \ref{fig:8} and \ref{fig:10}, whereas the results shown in Fig. \ref{fig:9} and \ref{fig:11} are calculated by an interior point algorithm which results typically not in an edge of the polyhedron. Hence, every point on the segment between the points represented by Fig. \ref{fig:8} and \ref{fig:9}, and Fig. \ref{fig:10} and \ref{fig:11}, respectively, is also a solution with the same objective function value.

\begin{figure}
\hspace*{\fill}
\subfigure{\label{fig:8}\thesubfigure\includegraphics[width=0.21\textwidth,trim=5mm 7mm 7mm 7mm,clip]{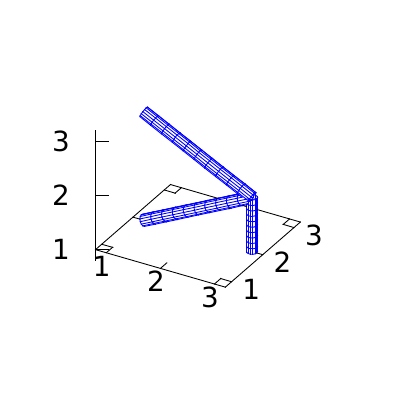}}
\hspace*{\fill}
\subfigure{\label{fig:9}\thesubfigure\includegraphics[width=0.21\textwidth,trim=5mm 7mm 7mm 7mm,clip]{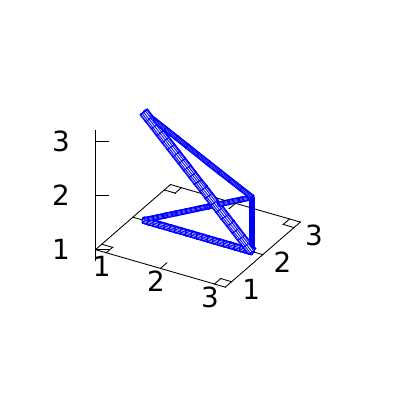}}
\hspace*{\fill}
\caption{Optimal solutions of example $3$. \subref{fig:8} Simplex algorithm. \subref{fig:9} Interior point algorithm.}
\label{fig:8,9}
\end{figure}

\begin{figure}
\hspace*{\fill}
\subfigure{\label{fig:10}\thesubfigure\includegraphics[width=0.21\textwidth,trim=5mm 7mm 7mm 7mm,clip]{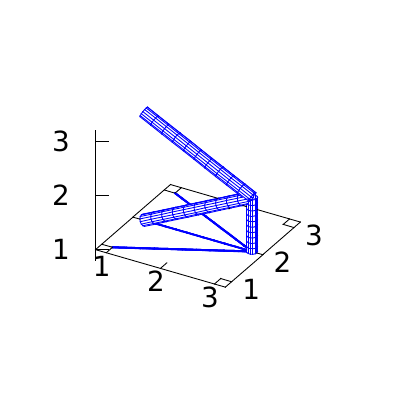}}
\hspace*{\fill}
\subfigure{\label{fig:11}\thesubfigure\includegraphics[width=0.21\textwidth,trim=5mm 7mm 7mm 7mm,clip]{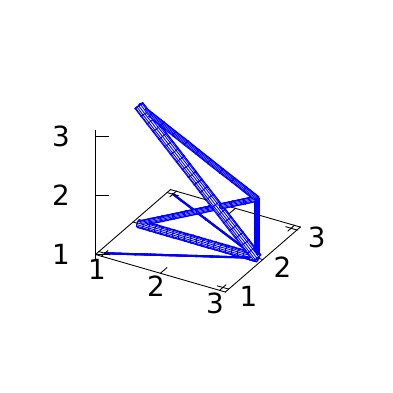}}
\hspace*{\fill}
\caption{Robust optimal solutions of example $3$. \subref{fig:10} Simplex algorithm. \subref{fig:11} Interior point algorithm.}
\label{fig:10,11}
\end{figure}

\paragraph{Example $4$:}
In this example of a mast are $72$ nodes $\left(a,b,c\right) \in \R^3$, $a,b \in \{0,\frac{1}{7},\frac{2}{7}\}$, $c \in \{\frac{x}{7}:x=0,1,\ldots,7\}$ measured in meter, and $503$ potential bars between every two not fixed nodes with a maximal distance of $\frac{\sqrt{3}}{7}$ --- long bars which are located alongside several shorter bars are ignored. The supports are in the nodes $\left(a,b,0\right)$, $a,b \in \{0,\frac{2}{7}\}$, which are fixed in every direction. In the node $\left(\frac{1}{7},\frac{1}{7},1\right)$ acts a force of $4\cdot10^4 \, \mbox{N}$ in the negative $z$-direction. The optimal solution with a volume of $0.000514 \, \mbox{m}^3$ is presented in Fig. \ref{fig:12}, \ref{fig:12pov}. For a perturbation of $50 \, \mbox{\%}$ in every direction of the force and scaling to have the same maximal quantities as in the unperturbed case, our robust optimization method produces the robust optimal solution shown in Fig. \ref{fig:13}, \ref{fig:13pov} with a volume of $0.001568 \, \mbox{m}^3$.

\begin{figure}
\hspace*{\fill}
\subfigure{\label{fig:12}\thesubfigure\includegraphics[width=0.21\textwidth,trim=5mm 7mm 7mm 4mm,clip]{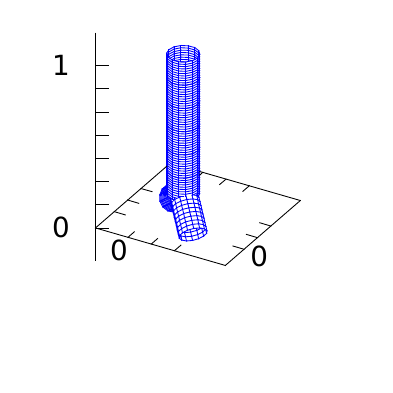}}
\hspace*{\fill}
\subfigure{\label{fig:13}\thesubfigure\includegraphics[width=0.21\textwidth,trim=5mm 7mm 7mm 4mm,clip]{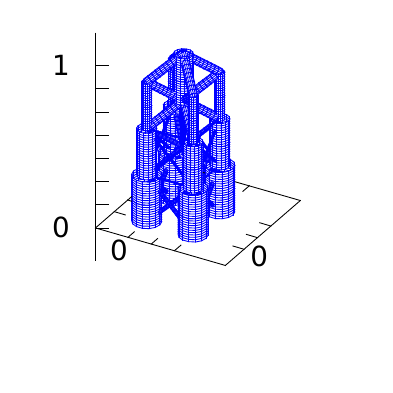}}
\hspace*{\fill}
\caption{\subref{fig:12} Optimal solution of example $4$. \subref{fig:13} Robust optimal solution of example $4$.}
\label{fig:12,13}
\end{figure}

\begin{figure}
\hspace*{\fill}
\subfigure{\label{fig:12allpov}\thesubfigure\includegraphics[width=0.3\textwidth]{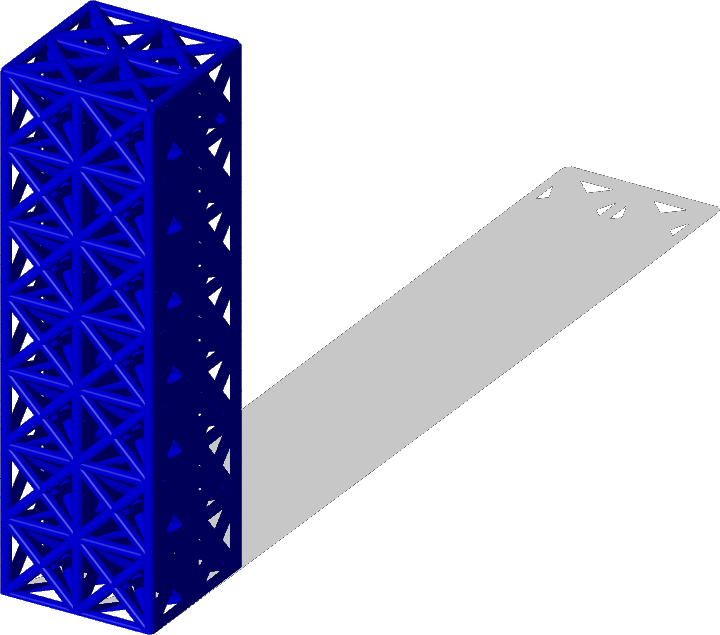}\hspace*{-3.5cm}}
\hspace*{\fill}
\subfigure{\label{fig:12pov}\thesubfigure\includegraphics[width=0.3\textwidth]{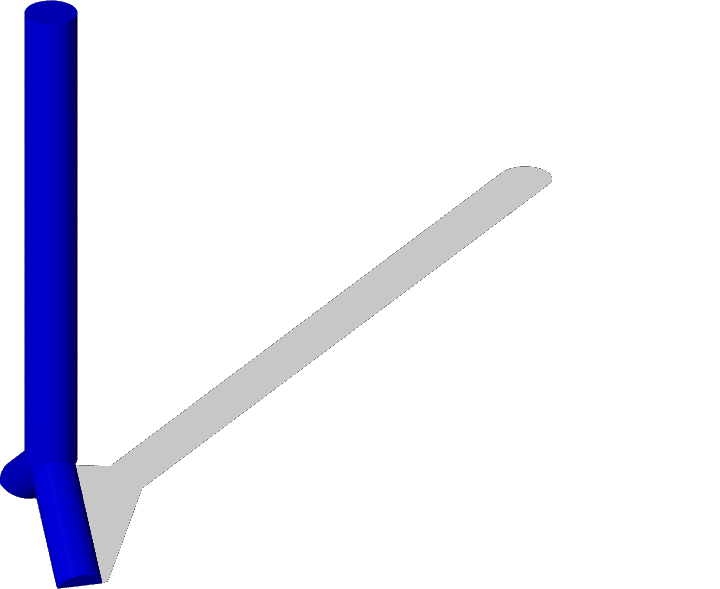}\hspace*{-4.2cm}}
\hspace*{\fill}
\subfigure{\label{fig:13pov}\thesubfigure\includegraphics[width=0.3\textwidth]{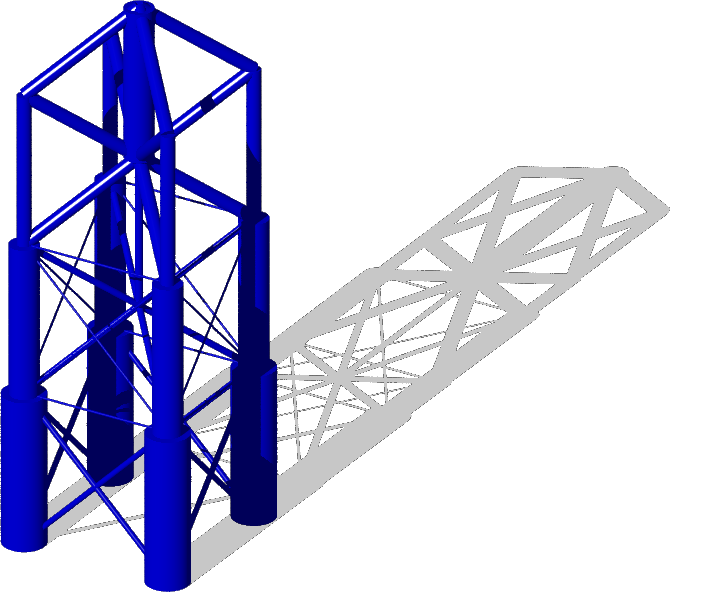}\hspace*{-2cm}}
\hspace*{\fill}
\caption{\subref{fig:12allpov} Ground structure of example $4$. \subref{fig:12pov} Optimal solution of example $4$. \subref{fig:13pov} Robust optimal solution of example $3$.}
\label{fig:12,13pov}
\end{figure}

\paragraph{Example $5$:}
As a last example, we show a multiple load case. It represents the top part of a transmission tower which supports two ground wires. The $27$ nodes $\left(a,b,c\right) \in \R^3$, $a,b,c \in \left\{0,\frac{1}{2},1\right\}$ are connected by $294$ potential bars between every two not fixed nodes --- long bars which are located alongside several shorter bars are ignored. The $4$ nodes $\left(\frac{1}{2},0,0\right),\left(\frac{1}{2},1,0\right),\left(0,\frac{1}{2},0\right),\left(1,\frac{1}{2},0\right) \in \R^3$ are fixed in every direction. We have $2$ forces $f_1$ and $f_2$ acting alone or together constituting the $3$ load cases --- represents the existing of either each single or both ground wires on top of the transmission tower. $f_1$ and $f_2$ are acting in $\left(0,0,1\right)$ and $\left(1,1,1\right)$, respectively, in the negative z-direction with a magnitude of $2\cdot10^4 \, \mbox{N}$.

The optimal solution with a volume of $0.000850 \, \mbox{m}^3$ is presented in Fig. \ref{fig:14}. For a perturbation of $50 \, \mbox{\%}$ in every direction of the forces, and scaling to obtain the same maximal quantities as in the unperturbed case, our robust optimization method produces the robust optimal solution shown in Fig. \ref{fig:15}, \ref{fig:16} with a volume of $0.002562 \, \mbox{m}^3$. Fig. \ref{fig:15} is the result of a simplex algorithm and therefore an edge of the feasible region, whereas Fig. \ref{fig:16} is calculated by an interior point algorithm.

\begin{figure*}
\hspace*{\fill}
\subfigure{\label{fig:14}\thesubfigure\hspace*{-1cm}\includegraphics[width=0.45\textwidth]{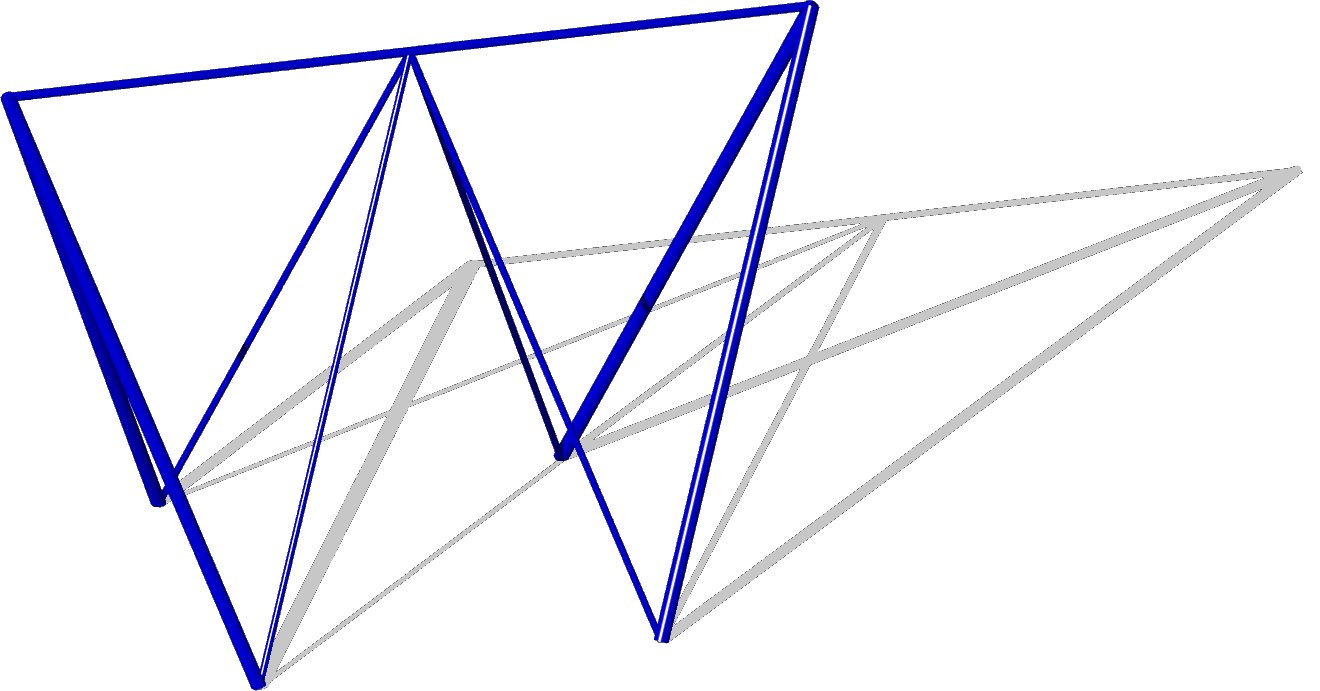}}\hspace*{-2cm}
\hspace*{\fill}
\subfigure{\label{fig:15}\thesubfigure\hspace*{-1cm}\includegraphics[width=0.45\textwidth]{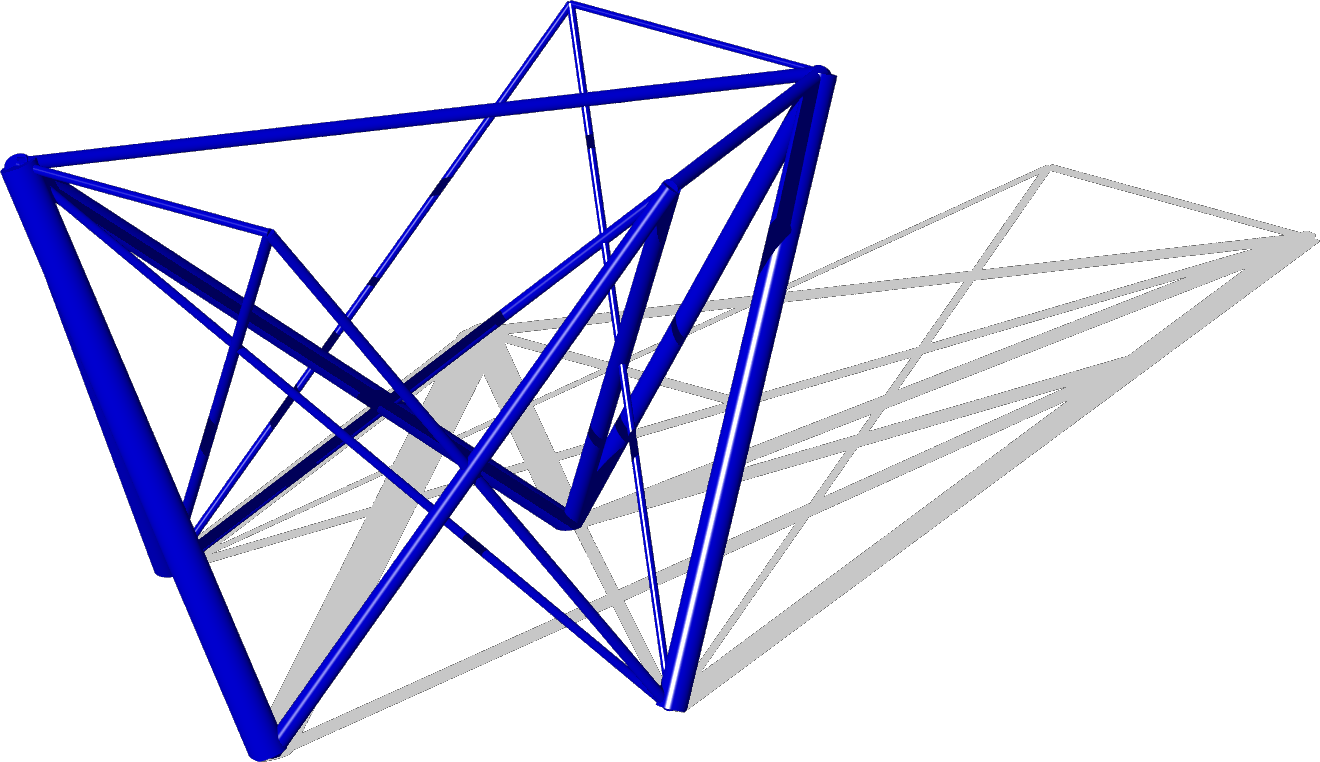}}\hspace*{-2cm}
\hspace*{\fill}
\subfigure{\label{fig:16}\thesubfigure\hspace*{-1cm}\includegraphics[width=0.45\textwidth]{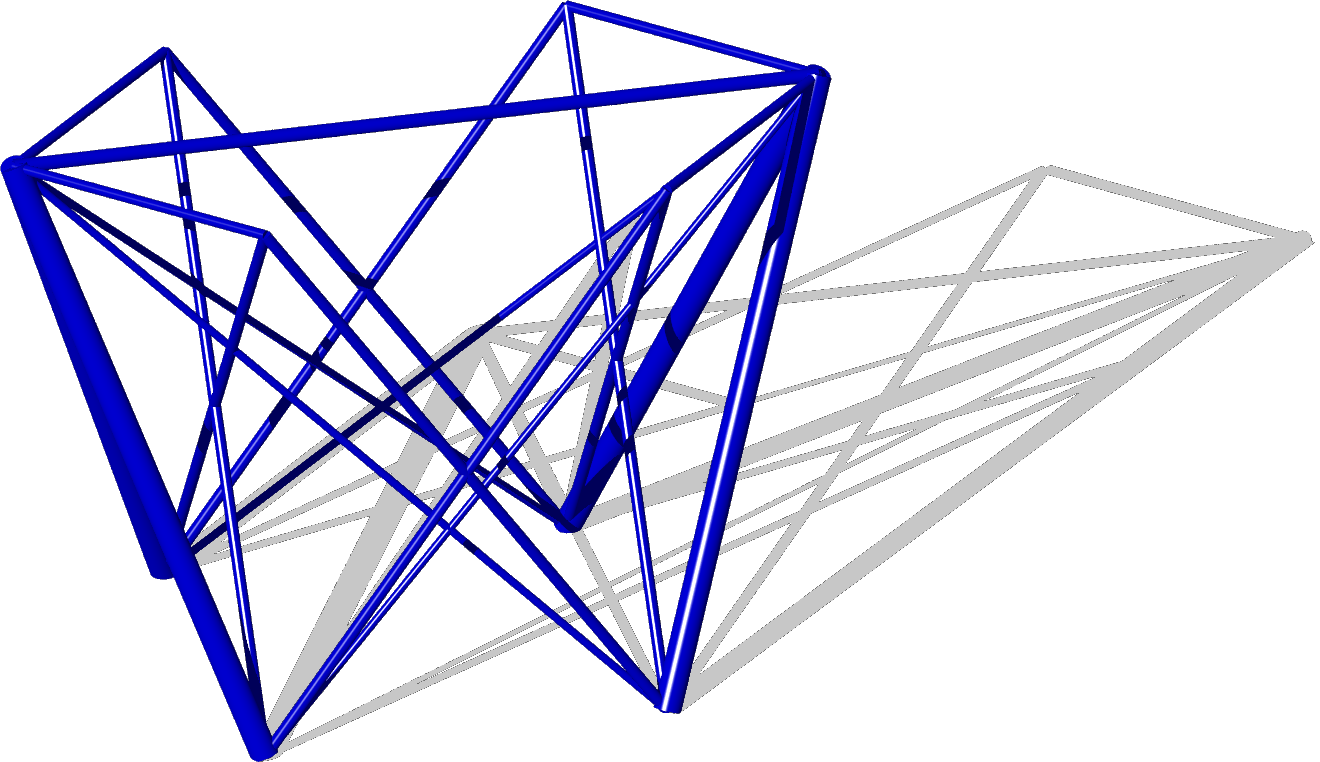}}\hspace*{-2cm}
\hspace*{\fill}
\caption{\subref{fig:14} Optimal solution of the example $5$. \subref{fig:15}, \subref{fig:16} Robust optimal solution of example $5$.}
\label{fig:14-16}
\end{figure*}

\section{Conclusion}
In the introduction we became acquainted with the origins of uncertainties. In our opinion every treatment of uncertainties leads in the end to robust optimization. The bad condition of the topology of a truss makes it necessary to choose artificial limits to definitely get a feasible solution for the analytical problem from the computer. In contrast to the common finite element approach of topology optimization, the basic initial approach chosen here is able to minimize the volume. The robust optimization yields a semi-infinite optimization problem. Typical commercial software products are able to handle semi-infinite optimization problems by adaptive generation of finite approximations of the infinite constraints. The expert knowledge of an engineer is sometimes also able to generate a wise approximation. Every time, these discrete approximations lead to the structure of the multiple load case. If we are looking for a black-box computer program to solve our problems, we need a well-defined and straightforward procedure. Our approach is able to fulfill this dream without an approximation. Finally, the numerical examples show the advantage of the robust optimization with the only drawback of a slight increase of the volume.

\bibliographystyle{plainnat}
\bibliography{literatur}
\end{document}